\documentclass[a4paper,11pt]{article}

\usepackage{amssymb}

\usepackage[cmex10]{amsmath}
\usepackage{amsfonts}
\usepackage{amsthm}

\usepackage{url}

\usepackage{fixltx2e}

\usepackage{url}	  

\usepackage{lmodern}
\usepackage[T1]{fontenc}
\usepackage{textcomp}

\theoremstyle{definition} \newtheorem{theorem}{Theorem}[section]
\theoremstyle{definition} \newtheorem{definition}[theorem]{Definition}
\theoremstyle{definition} \newtheorem{lemma}[theorem]{Lemma}
\theoremstyle{definition} \newtheorem{proposition}[theorem]{Proposition}
\theoremstyle{definition} \newtheorem{corollary}[theorem]{Corollary}
\theoremstyle{definition} 
\theoremstyle{definition} 
\theoremstyle{definition} 
\theoremstyle{definition} \newtheorem{remark}[theorem]{Remark}
\theoremstyle{definition} 
\newcommand{\eat}[1]{}

\begin{document}

\title{An Algebra of Fault Tolerance}
\date{}
\author{\begin{tabular}[t]{c@{\extracolsep{2em}}c} 
    Shrisha Rao\footnote{\tt shrao@alumni.cmu.edu} \\
    International Institute of Information Technology - Bangalore \\ 
    Bangalore 560 100 \\ India
\end{tabular}}
\maketitle

\begin{abstract}

  Every system of any significant size is created by composition from
  smaller sub-systems or components.  It is thus fruitful to analyze
  the fault-tolerance of a system as a function of its composition.
  In this paper, two basic types of system composition are described,
  and an algebra to describe fault tolerance of composed systems is
  derived.  The set of systems forms monoids under the two composition
  operators, and a semiring when both are concerned.  A partial
  ordering relation between systems is used to compare their
  fault-tolerance behaviors.  

\end{abstract}

{\bf Keywords:} systems, composition, fault tolerance, algebra,
semirings

{\bf 2000 MSC:} 68M15, 16Y60

\section{Introduction}

Fault tolerance is a subject of immense importance in the design and
analysis of many kinds of systems.  It has been studied in distributed
computing~\cite{JCT92,J94,Gartner99} and elsewhere in computer
science~\cite{Abb90,arora1992thesis}, in the context of
safety-critical systems~\cite{Rush94,MIT97}, and in many other places.
Authors such as Perrow~\cite{Perrow99} and Neumann~\cite{Neu95} note
many instances where standard assumptions made in designing
fault-tolerant systems seem not to hold, and draw conclusions from
these deviations.

In this paper, we take a different, algebraic view of fault tolerance,
based on the composition of a system.  The basic notion is that every
system of any significant size is created by composition from smaller
sub-systems or components.  The fault-tolerance of the overall system
is influenced by that of the components that underlie it.  Conversely,
a system may itself be a module or part of a larger system, so that
its fault-tolerance affects that of the whole of which it is a part.

It is thus fruitful to analyze the fault tolerance of a system as a
function of its composition.  In the preliminary analysis, it is
assumed that component failures occur independently---failure of one
component does not automatically imply the failure of another.

Composition of components to create a larger system is considered in
Section~\ref{preliminaries} to happen in two ways: direct sum, denoted
\(+\), and direct product, denoted \(\times\).  This is then used to
describe an arithmetic on systems.  The fault tolerance of systems is
formally described in Section~\ref{faulttolerance} using functions
from the set of systems to the set of natural numbers, whose basic
properties are obtained.  In Section~\ref{monoid}, the arithmetic on
systems is extended to define system monoids by direct sum and direct
product, and later, a system semiring considering both.  Using this as
a basis, a partial ordering of systems by fault tolerance is given in
Section~\ref{partialordering}, and basic properties are derived.  In
Section~\ref{eqclasses}, it is shown that monoids can be defined on
the set of equivalence classes under fault tolerance, of systems
(rather than on the set of systems as done in Section~\ref{monoid}).
Section~\ref{eqclasssemirings} briefly indicates how this could be
used to create a semiring on the set of equivalence classes of systems
under worst-case fault tolerance, essentially allowing for the
development of results analogous to those of
Section~\ref{partialordering}.  

The mathematics used is standard algebra as might be learned in a
first-level graduate or upper-level undergraduate course, and is
mostly self-explanatory.  Standard works in the field such as
Hungerford~\cite{Hungerford} or Lang~\cite{Lang2002}, and monographs
on semiring theory~\cite{Golan1999, HebWein1998} may be useful
references.

\section{Notation and Preliminaries} \label{preliminaries}

We use lowercase $p$, $q$, etc., to denote individual components that
are assumed to be \emph{atomic} at the level of discussion, i.e., they
have no components or sub-systems of their own.  Systems that are not
necessarily atomic are denoted by \(A, B\), etc., with or without
subscripts.  Where particular clarity is required, \emph{component}
will be used to refer to an atomic component, while \emph{subsystem}
will be used to refer to a component that is not atomic.  Sets of
systems---atomic or not--- are denoted by \(\mathcal{P},
\mathcal{Q}\), etc., and in particular \(\mathcal{U}\) is the
universal set of all systems in the domain of discourse.  The set of
natural numbers is denoted by \(\mathbb{N}\), and the set of integers
by \(\mathbb{Z}\).  Other mathematical terms and notation are defined
in the sections where they are first found.

In the beginning, we assume that our components and subsystems are
\emph{disjoint}, i.e., that they do not share any components among
themselves, and that they fail, if they do, independently of one
another.  

A \emph{fault} is a failure of a subsystem or component, while a
\emph{failure} applies to some system as a whole, but the latter term
will also be applied in the context of events that are faults from a
larger system perspective but may be considered failures from a
component or subsystem perspective.

Let $A$ and $B$ be two components that can form a system.  

\begin{definition} \label{def_sum}

The $+$ operator is considered to apply for systems consisting of two
components when the failure of either would cause the system as a
whole to fail.  Equivalently, we may say that the direct sum \(A + B\)
of components $A$ and $B$ is a system where the failure of either $A$
or $B$ will cause overall system failure.

\end{definition}

One example is a computer program that needs two resources, such as
two files, in order to perform, with lack of functionality resulting
from the lack of availability of either.

A more basic example of such a system is two light-bulbs connected in
series, with a voltage applied to cause them to glow.  If either bulb
burns out, the connection is broken and system failure occurs.

In colloquial terms, a direct sum is a situation where the ``weakest
link'' component needs to fail, for the system to fail.

\begin{definition} \label{def_product}

The $\times$ operator applies for systems consisting of two components
when the failure of both is necessary to cause the sytem as a whole to
fail.  Equivalently, we may say that the direct product \(A \times B\)
of components $p$ and $q$ is a system where the failure of both is
necessary to cause overall system failure.

\end{definition}

One example is a computer program that needs a file that is available
in two replicas; the program would only fail to perform if both
replicas were to be unavailable.

A more basic example of such a system is one consisting of two
light-bulbs connected in parallel across a voltage source, with the
system considered functional if at least one bulb glows.  System
failure occurs only when both bulbs burn out.

In colloquial terms, a direct product is a situation where the
``strongest link'' component needs to fail, for the system to fail.

Note that these composition operators are different from the way
composition is described elsewhere in the literature~\cite{AL93,AL95},
and are also not significantly related to the theory of ``fault
tolerance components'' proposed by Arora and
Kulkarni~\cite{arora1998icdcs}.

Given the previous definitions of the direct sum and direct product
operators, we obtain an ``arithmetic'' on systems as
follows~\cite{shrao}, where \(=\) stands for isomorphism.

\begin{proposition} \label{opsproperties}

Given disjoint components $A$, $B$, and $C$, we have the following.

\begin{itemize}

\item[(i)] $+$ and $\times$ are commutative and associative: \(A
  \times B = B \times A\); \(A + C = C + A\); \(A + (B + C) = (A + B)
  + C\); and \(A \times (B \times C) = (A \times B) \times C\).

\item[(ii)] $\times$ distributes over $+$: \(A \times (B + C) = (A
  \times B) + (A \times C)\).

\end{itemize}
\end{proposition}

\begin{proof}

  The properties given in $(i)$ are obvious from the
  Definitions~\ref{def_sum} and~\ref{def_product}.

  For the distributive property of $(ii)$, consider that \(A \times
  (B+C)\) will fail when both $A$ and \(B+C\) fail, and that this
  happens precisely when \((A \times B) + (A \times C)\)
  fails. \end{proof}

Therefore, systems are described by polynomial expressions on
variables denoting their components, such as \((P^3 + 3Q) \times 5R^2\).

In addition, by analogy with regular arithmetic, we can say the
following.

\begin{proposition} \label{sumofproducts}

Any system described by an expression of direct sum and product
operations on components \(A_{i,j}\) can be expressed in a canonical
\emph{sum of products} form such as \(\{A_{0,0} \times A_{0,1} \times
\ldots \times A_{0,n_0 - 1}\} + \{A_{1,0} \times A_{1,1} \times \ldots
\times A_{1,n_1 - 1}\} + \{A_{m - 1,0} \times A_{m - 1,1} \times
\ldots \times A_{m - 1, n_{m-1} -1}\}\).

\end{proposition}

In Proposition~\ref{sumofproducts}, there are $m$ product terms to be
added together, and a product term $i$ consists of the direct product
of $n_i$ components.

\section{Fault Tolerance of Composed Systems} \label{faulttolerance}

About systems that are members of \(\mathcal{U}\), we can state
the following results from~\cite{shrao}, with \(A^n\) describing the
direct product of \(A \in \mathcal{U}\) with itself $n$ times, and
\(nA\) describing the direct sum of \(A \in \mathcal{U}\) with itself
$n$ times.

\begin{lemma}

  A system \(A^n\) can tolerate faults in \(n - 1\) of the subsystems,
  where \(n \in \mathbb{Z}_+\), and a system \(nA\) can tolerate zero
  faults.

\end{lemma}

\begin{proof}

  This is immediate, given the definitions of the direct sum and
  product.  A system composed of the direct sum of $A$ with itself $n$
  times will fail should any one of the components fail, whilst one
  composed of the direct product of $A$ with itself $n$ times will
  only fail if all of them fail (i.e., it can tolerate faults in \(n -
  1\) of them). \end{proof}

As a consequence, we have the following.

\begin{corollary} \label{theoremmpn1}

  A system \(mA^n\) consists of \(m \times n\) subsystems, where \(m,n
  \in \mathbb{Z}_+\), but can only tolerate faults in \(n - 1\) of
  them in the worst case.

\end{corollary}

Notice, however, that the following holds.

\begin{corollary} \label{theoremmpn2}

  In the best case a system \(mA^n\) can tolerate faults in \(m \times
  (n - 1)\) components.

\end{corollary}

\begin{proof}

  In the best case, a system \(mA^n = A^n + \ldots + A^n\) would have
  \((n - 1)\) failures in each of the \(A^n\) and still be functional,
  for a tolerance of \(m \times (n - 1)\) faults. \end{proof}

In the previous results, we have dealt with just one type of component.
However, the same idea can be generalized to multiple types of
components, as follows.  It is assumed here that a failure of any
component, regardless of its type, is of equal value and adds 1 to the
number of failures in the overall system.

The idea is to specify the fault-tolerance (in the best case and in
the worst case) recursively, in terms of sub-systems.

\begin{definition} \label{def_faulttolerance}

  If, as stated, previously, $\mathcal{U}$ is the set of all systems,
  then \(\Phi_{best}: \mathcal{U} \rightarrow \mathbb{N} \) and
  \(\Phi_{worst}: \mathcal{U} \rightarrow \mathbb{N}\) are functions
  giving the best- and worst-case fault-tolerances of a system, such
  that:

\begin{itemize}

\item[(i)] \(\Phi_{best}(A) = i\) for \(A \in \mathcal{U}\), if $i$ is
  the cardinality of the largest set of components of $A$ that can
  fail without causing $A$ to fail.

\item[(ii)] \(\Phi_{worst}(B) = j\) for \(B \in \mathcal{U}\), if
  \(j+1\) is the cardinality of the smallest set of components in $B$
  whose failure causes $B$ to fail.

\end{itemize}

\end{definition}

Quite obviously, we have \(\Phi_{best}(A) \geq \Phi_{worst}(A),
\forall A \in \mathcal{U}\).  The values of these functions may be
computed recursively as indicated in the following theorem.

\begin{theorem} \label{theoremfaulttolerance}

  Given systems \(A, B \in \mathcal{U}\), if $|A|$ and $|B|$ are the
  numbers of components therein, the following hold:

\begin{itemize}

\item[(i)] \(\Phi_{best}(A \times B) = \max\{\Phi_{best}(A) + |B|, |A|
  + \Phi_{best}(B)\}\);
\item[(ii)] \(\Phi_{worst}(A \times B) = \Phi_{worst}(A) +
  \Phi_{worst}(B) + 1\);
\item[(iii)] \(\Phi_{best}(A + B) = \Phi_{best}(A) + \Phi_{best}(B)\);
\item[(iv)] \(\Phi_{worst}(A + B) = \mathrm{min}\{\Phi_{worst}(A),
  \Phi_{worst}(B)\}\) .

\end{itemize}

\end{theorem}

\begin{proof}

  For parts $(i)$ and $(ii)$, we know that \(A \times B\) will fail only
  when both $A$ and $B$ fail.

  Therefore, \(\Phi_{best}(A \times B)\) reflects the case where one
  of $A$ or $B$ fails fully (every component in it failing), and the
  other has its best-case limit of component failures.  Therefore, the
  total number is \(\max\{\Phi_{best}(A) + |B|, |A| +
  \Phi_{best}(B)\}\) failures.  Likewise, \(\Phi_{worst}(A \times B) =
  \Phi_{worst}(A) + \Phi_{worst}(B) + 1\).

  For parts $(iii)$ and $(iv)$, we know that a system \(A + B\) will
  fail should either $A$ or $B$ fail.

  Therefore, \(\Phi_{best}(A + B) = \Phi_{best}(A) + \Phi_{best}(B)\),
  because in the best case $A$ and $B$ will sustain their limit of
  failures and yet not fail.

  Likewise, \(\Phi_{worst}(A + B) = \mathrm{min}\{\Phi_{worst}(A),
  \Phi_{worst}(B)\}\), as in the worst case the minimum of
  \(\mathrm{min}\{\Phi_{worst}(A), \Phi_{worst}(B)\}\) is all that is
  necessary to cause one of the two, $A$ or $B$, to fail. \end{proof}

The theorem just stated can also be extended in the obvious way to
systems of \(n > 2\) components.  

\begin{theorem} \label{corollaryfaulttolerance}

  Given subsystems \(A_i \in \mathcal{U}, 0 \leq i \leq n-1\), we have
  the following, with \(\Pi\) and \(\Sigma\) denoting, respectively,
  the direct product and direct sum of multiple subsystems, and
  \(|A_i|\) denoting the number of components in subsystem \(A_i\):

\begin{itemize}

\item[(i)] For the best-case tolerance of a product of $n$ subsystems:
\[\Phi_{best}(\prod^{n-1}_{i=0} A_i) = \sum^{n-1}_{i=0}|A_i| - |A_k| + \Phi_{best}(A_k),\] 

where $A_k$ is a subsystem with the lowest difference between its own
number of components and its own best-case fault tolerance.

\item[(ii)] For the worst-case tolerance of a product of $n$ subsystems:
\[\Phi_{worst}(\prod^{n-1}_{i=0} A_i) = \sum^{n-1}_{i=0}\Phi_{worst}(A_i) + n - 1\]

\item[(iii)] For the best-case tolerance of a sum of $n$ subsystems:
\[\Phi_{best}(\sum^{n-1}_{i=0} A_i) = \sum^{n-1}_{i=0} \Phi_{best}(A_i)\]

\item[(iv)] For the worst-case tolerance of a sum of $n$ subsystems:
\[\Phi_{worst}(\sum^{n-1}_{i=0} A_i) = \min\{\Phi_{worst}(A_i)\}\]

\end{itemize}

\end{theorem}

\begin{proof}

  The proof is exactly similar to that of
  Theorem~\ref{theoremfaulttolerance}, though a little more involved.

  For parts $(i)$ and $(ii)$, we know that the system will fail when
  all of the $A_i$ subsystems fail.

  Therefore, the best-case fault tolerance of the product of the
  \(A_i\) is when all the subsystems save one fail, and collectively
  sustain the maximum possible number of component failures.  If the
  one that does not fail also sustains component failures, then the
  best-case fault tolerance is obtained when the non-failing subsystem
  has the lowest difference, among all subsystems $A_i$, between its
  own number of components and its best-case fault tolerance.

  Likewise, the worst-case fault tolerance of the product of the
  \(A_i\) is when all subsystems but one fail, but sustain the least
  possible number of component failures in doing so.  This happens
  when each $A_i$ sustains its worst-case number of failures, and then
  further when \(n - 1\) of them sustain one more failure each,
  meaning that all but one subsystem have failed.

  For parts $(iii)$ and $(iv)$, we know that the system will fail when
  any one subsystem \(A_i\) fails.

  Therefore, the best-case fault tolerance of the sum of the \(A_i\)
  is when each $A_i$ sustains its best-case limit of failures.
  Likewise, the worst-case fault tolerance of the sum of the \(A_i\)
  is when just one $A_i$ sustains its worst-case limit of
  failures. \end{proof}

\section{System Monoids and Semirings} \label{monoid}

Given the system arithmetic previously defined, we can posit the
existence of two identity operators, one each for the direct sum and
direct product.  Informally, an identity element is one which leaves
the system unchanged, under the relevant operator.

\begin{definition}

  The multiplicative and additive identities are defined as follows.

\begin{itemize}

\item[(i) ] The additive identity \(\mathbf{0}\) is the
  system such that for any system $A$, \(A + \mathbf{0} =
  \mathbf{0} + A = A\).

\item[(ii) ] The multiplicative identity \(\mathbf{1}\) is
  the system such that for any system $A$, \(A \times
  \mathbf{1} = \mathbf{1} \times A = A\).

\end{itemize}

\end{definition}

By the commutativity of the $+$ and $\times$ operators, we observe
that the identity elements are two-sided.

Informally, we may describe these elements as follows:

\begin{itemize}

\item[(i)] The additive identity \(\mathbf{0}\) is a
  system ``that is always up.''  The direct sum of such a system and
  $A$ is obviously $A$ itself.

\item[(ii)] The multiplicative identity \(\mathbf{1}\) is
  a system ``that is always down.''  The direct product of such a
  system and $A$ is likewise $A$ itself.

\end{itemize}

Then $\mathcal{U}$, combined with the $+$ operator, is a \emph{monoid}
(a set with an associative operator and a two-sided identity
element)~\cite{Hungerford}.  Similarly, $\mathcal{U}$ is also a monoid
when considering the $\times$ operator.  For notational convenience,
we denote these monoids as \((\mathcal{U},+)\) and
\((\mathcal{U},\times)\).

It is further clear that the set $\mathcal{U}$ is a \emph{semiring}
when taken with the operations $+$ and $\times$ because the following
conditions~\cite{Golan1999} for being a semiring are satisfied:

\begin{itemize}

\item[(i)] \((\mathcal{U}, +)\) is a commutative monoid with identity
  element \(\mathbf{0}\);

\item[(ii)] \((\mathcal{U}, \times)\) is a monoid with identity
  element \(\mathbf{1}\);

\item[(iii)] \(\times\) distributes over \(+\) from either side;

\item[(iv)] \(\mathbf{0} \times A = \mathbf{0} = A \times \mathbf{0}\)
  for all \(A \in \mathcal{U}\).

\end{itemize}

This system semiring will be denoted by \((\mathcal{U},+,\times)\),
and its properties are as indicated in the following.

\begin{remark} \label{remark1}

  The semiring $(\mathcal{U},+,\times)$ is \emph{zerosumfree}, because
  \(A + B = \mathbf{0}\) implies, for all \(A, B \in \mathcal{U}\),
  that \(A = B = \mathbf{0}\).

\end{remark}

This condition shows~\cite{Golan1999} that the monoid
\((\mathcal{U},+)\) is completely removed from being a group, because
no non-trivial element in it has an inverse.

\begin{remark} \label{remark2}

  $(\mathcal{U},+,\times)$ is \emph{entire}, because there are no
  non-zero elements \(A, B \in \mathcal{U}\) such that \(A \times B =
  \mathbf{0}\).

\end{remark}

This likewise shows that the monoid \((\mathcal{U},\times)\) is
completely removed from being a group, as there is no non-trivial
multiplicative inverse.

\begin{remark} \label{remark3}

  $(\mathcal{U},+,\times)$ is \emph{simple}, because \(\mathbf{1}\) is
  \emph{infinite}, i.e., \(A + \mathbf{1} = \mathbf{1}, \forall A \in
  \mathcal{U}\).

\end{remark}

Given that the semiring \((\mathcal{U},+,\times)\) is both zerosumfree
and entire, we can call it an \emph{information
  algebra}~\cite{Golan1999}.

We may state another important definition~\cite{Golan1999} about
semirings, and observe a property of \((\mathcal{U},+,\times)\).

\begin{definition} \label{def_center}

  The \emph{center} \(C(\mathcal{U})\) of $\mathcal{U}$ is the set
  \(\{A \in \mathcal{U} \ | \ A \times B = B \times A, \ \mathrm{for \
    all} \ B \in \mathcal{U}\}\).

\end{definition}

\begin{remark} \label{remark4}

  The semiring $(\mathcal{U},+,\times)$ is \emph{commutative} because
  \(C(\mathcal{U}) = \mathcal{U}\).

\end{remark}

\section{Fault-Tolerance Partial Ordering} \label{partialordering}

Consider a partial ordering relation \(\preccurlyeq\) on
\(\mathcal{U}\), the set of all systems, such that \((\mathcal{U},
\preccurlyeq)\) is a poset.  This is a \emph{fault-tolerance partial
  ordering} where \(A \preccurlyeq B\) means that $A$ has a lower
measure of some fault metric than $B$ (e.g., $A$ has fewer failures
per hour than $B$, or has a better fault tolerance than $B$).

Formally, $\preccurlyeq$ is a partial ordering on the semiring
$(\mathcal{U},+,\times)$ if the following conditions are
satisfied~\cite{HebWein1998}.

\begin{definition} \label{defpartialordering}

  If $(\mathcal{U},+,\times)$ is a semiring and \((\mathcal{U},
  \preccurlyeq)\) is a poset, then
  \((\mathcal{U},+,\times,\preccurlyeq)\) is a \emph{partially ordered
    semiring} if the following conditions are satisfied for all \(A,
  B,\) and $C$ in $\mathcal{U}$.

\begin{itemize}

\item[(i)] The \emph{monotony law of addition}:

\[A \preccurlyeq B \longrightarrow A + C \preccurlyeq B + C\]

\item[(ii)] The \emph{monotony law of multiplication}:

\[A \preccurlyeq B \longrightarrow A \times C \preccurlyeq B \times C.\]

\end{itemize}

\end{definition}

It is assumed that \(\mathbf{0} \preccurlyeq A, \forall A \in
\mathcal{U}\), on the reasoning that \(\mathbf{0}\), being a system
that never fails, must have the least possible measure of any fault
metric.  Similarly, \(A \preccurlyeq \mathbf{1}\), because
\(\mathbf{1}\), being a system that is always down, must have the
greatest possible measure of any fault metric.

Given Definition~\ref{defpartialordering}, it is instructive to
consider the behavior of the partial order under composition.  We
begin with a couple of simple results.

\begin{lemma} \label{lemma1}

  If $\preccurlyeq$ is a fault-tolerance partial order as defined,
  then \(\forall A, B \in \mathcal{U}\):

\begin{itemize}

\item[(i)] \(A \preccurlyeq A + B\), and
\item[(ii)] \(A \times B \preccurlyeq A\).

\end{itemize}

\end{lemma}

\begin{proof}

  For $(i)$, consider that \(\mathbf{0} \preccurlyeq B\).  Using the
  monotony law of addition, we get \(\mathbf{0} + A \preccurlyeq B +
  A\).  Considering that $\mathbf{0}$ is the additive identity element
  and that addition is commutative, we get \(A \preccurlyeq A + B\).

  For $(ii)$, consider that \(B \preccurlyeq \mathbf{1}\).  Using the
  monotony law of multiplication, we get \(B \times A \preccurlyeq
  \mathbf{1} \times A\).  Considering that $\mathbf{1}$ is the
  multiplicative identity element and that multiplication is
  commutative, we get \(A \times B \preccurlyeq A\). \end{proof}

A property of $\mathcal{U}$ in consideration of the $+$ operator can
now be noted.

\begin{remark} \label{positivecone}

  The \emph{positive cone} of \((\mathcal{U}, +, \preccurlyeq)\),
  which is the set of elements \(A \in \mathcal{U}\) for which \(A
  \preccurlyeq A + B, \forall B \in \mathcal{U}\), is the set
  $\mathcal{U}$ itself.  The \emph{negative cone} is empty.

\end{remark}

This is a direct consequence of Lemma~\ref{lemma1} $(i)$, and it also
follows that the set of elements \(\{A \, | \, A + B \preccurlyeq A\}
= \emptyset\), showing that the negative cone is empty.

The analogous property of $\mathcal{U}$ in consideration of the
$\times$ operator can also be noted.

\begin{remark} \label{negativecone}

  The \emph{negative cone} of \((\mathcal{U}, \times, \preccurlyeq)\),
  which is the set of elements \(A \in \mathcal{U}\) for which \(A
  \times B \preccurlyeq A, \forall B \in \mathcal{U}\), is the set
  $\mathcal{U}$ itself.  The \emph{positive cone} is empty.

\end{remark}

As before, this is a direct consequence of Lemma~\ref{lemma1} $(ii)$,
and it likewise also follows that the set of elements \(\{A \, | \, A
\preccurlyeq A \times B\} = \emptyset\), showing that the positive
cone is empty.

\begin{theorem} \label{sumconsistency}

  If \(A + B \preccurlyeq C\), then \(A \preccurlyeq C\) and \(B
  \preccurlyeq C\).

\end{theorem}

\begin{proof}

  The proof is by contradiction.  Assume the contrary.  Then \(A + B
  \preccurlyeq C\), and at least one of \(A \preccurlyeq C\) or \(B
  \preccurlyeq C\) is false.

  Without loss of generality, assume that \(C \preccurlyeq A\).  Using
  the monotony law of addition and the commutativity of the $+$
  operator, \(B + C \preccurlyeq A + B\).

  Now, by Lemma~\ref{lemma1} $(i)$, \(C \preccurlyeq B + C\).  Given
  the transitivity of $\preccurlyeq$, we get \(C \preccurlyeq A + B\),
  which is a contradiction.  \qedhere  

\end{proof}

An analogous result can also be stated in terms of the product, as
follows.

\begin{theorem} \label{productconsistency}

  If \(A \preccurlyeq B \times C\), then \(A \preccurlyeq B\) and \(A
  \preccurlyeq C\).

\end{theorem}

\begin{proof}

  The proof is again by contradiction.  Assume the contrary.  Then \(A
  \preccurlyeq B \times C\) and at least one of \(A \preccurlyeq B\)
  and \(A \preccurlyeq C\) is false.

  Without loss of generality, assume that \(B \preccurlyeq A\).  Using
  the monotony law of multiplication and the commutativity of the
  $\times$ operator, we get \(B \times C \preccurlyeq A \times C\).
  
  Now, by Lemma~\ref{lemma1} $(ii)$, \(A \times C \preccurlyeq A\).
  Given the transitivity of $\preccurlyeq$, we get \(B \times C
  \preccurlyeq A\), which is a contradiction. \qedhere

\end{proof}

The following generalization of Lemma~\ref{lemma1} can be made.

\begin{lemma} \label{lemma2}

If $\preccurlyeq$ is a fault-tolerance partial order, then \(\forall n \in
\mathbb{Z}_+\) and \(\forall A \in \mathcal{U}\),

\begin{itemize}

\item[(i)] \(A \preccurlyeq nA\), and
\item[(ii)] \(A^n \preccurlyeq A\).

\end{itemize}

\end{lemma}

\begin{proof}

  For $(i)$, consider that by Lemma~\ref{lemma1} $(i)$ we have \(A
  \preccurlyeq 2A\) (just set $B$ to be $A$ in the Lemma).  Likewise,
  \(kA \preccurlyeq (k+1)A\), for any \(k \geq 2\).  By the
  transitivity of $\preccurlyeq$, we therefore have the result.

  For $(ii)$, the reasoning is very similar, using Lemma~\ref{lemma1}
  $(ii)$ and the transitivity of $\preccurlyeq$. \qedhere

\end{proof}

The following is an obvious corollary.

\begin{corollary} \label{corollary1}

  If $\preccurlyeq$ is a fault-tolerance partial order, then \(\forall
  m,n \in \mathbb{Z}_+\) and \(\forall A \in \mathcal{U}\), if \(m <
  n\), we have:

\begin{itemize}

\item[(i)] \(mA \preccurlyeq nA\), and
\item[(ii)] \(A^n \preccurlyeq A^m\).

\end{itemize}

\end{corollary}

Similarly, the following corollary generalizing
Theorems~\ref{sumconsistency} and~\ref{productconsistency} applies.
The proof is omitted as obvious.

\begin{corollary} \label{genconsistency}

  The following hold for all \(A, B \in \mathcal{U}\) and all \(n \in
  \mathbb{Z}_+\):

\begin{itemize}

\item[(i)] if \(nA \preccurlyeq B\), then \(A \preccurlyeq B\); and
\item[(ii)] if \(A \preccurlyeq B^n\), then \(A \preccurlyeq B\).

\end{itemize}

\end{corollary}

In a system semiring, the following result concerning preservation of
`fault-tolerance behavior under composition also applies.

\begin{theorem} \label{theorempartialorder}

  If \(\preccurlyeq\) is a partial order as described, and if \(A
  \preccurlyeq B\) and \(C \preccurlyeq D\), then,

\begin{itemize}

\item[(i)] \(A + C \preccurlyeq B + D\), and
\item[(ii)] \(A \times C \preccurlyeq B \times D\).

\end{itemize}

\end{theorem}

\begin{proof}

These results can be proven directly.  Only $(i)$ is proved, the proof
of $(ii)$ being very similar.

We know the following:
\begin{eqnarray} \label{pleqq}
A \preccurlyeq B
\end{eqnarray}

and:
\begin{eqnarray} \label{rleqs}
C \preccurlyeq D
\end{eqnarray}

From~(\ref{pleqq}) and the monotony law of addition (considering the
direct sum of $D$ and both sides), we have:
\begin{eqnarray} \label{pq3}
A + D \preccurlyeq B + D.
\end{eqnarray}

Similarly, from~(\ref{rleqs}) and the monotony law (considering the
direct sum of $A$ and both sides), we have:
\begin{eqnarray} \label{rs4}
A + C \preccurlyeq A + D.
\end{eqnarray}

By considering transitivity in respect of (\ref{rs4}) and (\ref{pq3}),
we get:
\begin{eqnarray*}
A + C \preccurlyeq B + D.
\end{eqnarray*}

\emph{QED}. \qedhere

\end{proof}

The following corollary can be stated.

\begin{corollary} \label{corollary2}

  If $\preccurlyeq$ is a fault-tolerance partial order and if \(A
  \preccurlyeq B\), then \(\forall n \in \mathbb{Z}_+\),

\begin{itemize}

\item[(i)] \(nA \preccurlyeq nB\), and 
\item[(ii)] \(A^n \preccurlyeq B^n\).

\end{itemize}

\end{corollary}

\begin{proof}

  These are straight-forward iterated consequences of
  Theorem~\ref{theorempartialorder} and the transitivity of
  $\preccurlyeq$. \qedhere

\end{proof}

In consideration of the transitivity of the partial order
$\preccurlyeq$, the following generalization of
Theorem~\ref{theorempartialorder} applies.  The proof is immediate.

\begin{theorem} \label{theorempartialorder2}

  If \(A_i \preccurlyeq B_i\), with \(0 \leq i \leq n - 1\) and \(A_i,
  B_i \in \mathcal{U}\), then:

\[ \sum_{i=0}^{n-1} A_i \preccurlyeq \sum_{i=0}^{n-1} B_i, \]

and

\[ \prod^{n-1}_{i=0} A_i \preccurlyeq \prod^{n-1}_{i=0} B_i.\]

\end{theorem}

The following is a result along the same lines as the
previous two.

\begin{theorem} \label{theorem2}

    If \(A \preccurlyeq B\) and \(m \leq n\), then:

\begin{itemize}

\item[(i)] \(mA \preccurlyeq nB\); and
\item[(ii)] \(A^n \preccurlyeq B^m\).

\end{itemize}

\end{theorem}

Note that $A^m$ and $B^n$ cannot be compared just based on the data
given.

\begin{proof}

  For $(i)$, we have that \(mA \preccurlyeq mB\) by
  Corollary~\ref{corollary2} $(i)$, because \(A \preccurlyeq B\).  We
  further have that \(mB \preccurlyeq nB\) by
  Corollary~\ref{corollary1} $(i)$, because \(m \leq n\).  By the
  transitivity of $\preccurlyeq$, we have the desired result.

  For $(ii)$, we have that \(A^n \preccurlyeq A^m\) by
  Corollary~\ref{corollary1} $(ii)$, because \(m \leq n\).  We further
  have that \(A^m \preccurlyeq B^m\) by Corollary~\ref{corollary2}
  $(ii)$, because \(A \preccurlyeq B\).  As before, by the
  transitivity of $\preccurlyeq$, we have the desired
  result. \qedhere

\end{proof}

These results can be generalized in the obvious way, as follows.

\begin{theorem}

  If \(A_i \preccurlyeq A_{i+1}, n_i \leq n_{i+1}\), and \(m_i \geq
  m_{i+1}\), for some range of values $i$, then we have:

\[ n_i A_i^{m_i} \preccurlyeq n_{i+1} A_{i+1}^{m_{i+1}} \]

\end{theorem}

\begin{proof}

  It is clear that \(A_i^{m_i} \preccurlyeq A_{i+1}^{m_{i+1}}\),
  noting that \(A_i \preccurlyeq A_{i+1}\), \(m_i \geq m_{i+1}\), and
  applying Theorem~\ref{theorem2} $(ii)$.  Now, given that \(A_i^{m_i}
  \preccurlyeq A_{i+1}^{m_{i+1}}\), the fact that \(n_i \leq
  n_{i+1}\), in light of Theorem~\ref{theorem2} $(i)$, gives the
  result. \qedhere

\end{proof}

This theorem leads to the following obvious corollary.

\begin{corollary}

If \(m \leq n\), then \(m A^n \preccurlyeq n A^m\).

\end{corollary}

\section{Monoids of Fault-Tolerance Equivalence
  Classes} \label{eqclasses}

It has previously been noted in Section~\ref{monoid} that the set of
all systems, along with the direct sum (respectively: direct product)
and the additive (respectively: multiplicative) identity defines a
monoid.  Another kind of monoid may be defined on classes of systems,
as follows.

We define an equivalence relation on fault tolerance, as follows.

\begin{definition} \label{faultequiv}

  A fault-tolerance equivalence relation by worst-case fault
  tolerance, \(\mathcal{R} (\sim)\), is given by \(A \sim B
  \longrightarrow \Phi_{worst}(A) = \Phi_{worst}(B)\) for all \(A, B \in
  \mathcal{U}\).

\end{definition}

Based on this, we note a simple theorem, given in~\cite{shrao}.

\begin{theorem} \label{monoidlemma1}

\(A_1 \sim B_1\) and \(A_2 \sim B_2\) together imply:

\begin{itemize}

\item[(i)] \(A_1 + A_2 \sim B_1 + B_2\); and

\item[(ii)] \(A_1 \times A_2 \sim B_1 \times B_2\).

\end{itemize}

\end{theorem}

\begin{proof}

  For part (i), we note from Theorem~\ref{theoremfaulttolerance} (iv)
  that \(\Phi_{worst}(A_1 + A_2) = \mathrm{min}\{\Phi_{worst}(A_1) +
  \Phi_{worst}(A_2)\}\).  Let this be \(\Phi_{worst}(A_1)\).
  Similarly, then, \(\Phi_{worst}(B_1 + B_2) = \Phi_{worst}(B_1)\),
  and the result obtains.

  For part (ii), we note from Theorem~\ref{theoremfaulttolerance}
  (ii) that \(\Phi_{worst}(A_1 \times A_2) = \Phi_{worst}(A_1) +
  \Phi_{worst}(A_2) + 1\).  This is also equal to \(\Phi_{worst}(B_1)
  + \Phi_{worst}(B_2) + 1\) by the definition of $\mathcal{R}$, which
  in its turn is equal to \(\Phi_{worst}(B_1 \times B_2)\). \qedhere

\end{proof}

Now we can state the major result (given for monoids by
Hungerford~\cite{Hungerford}), on constructing a monoid on the
equivalence classes of systems.  The proof given here is as given by
Hungerford.

\index{direct sum}
\begin{theorem} \label{monoidtheorem1}

  Let \(\mathcal{R}(\sim)\) be defined as in
  Definition~\ref{faultequiv}.  Then the set
  \(\mathcal{U}/\mathcal{R}\) of all equivalence classes of
  \(\mathcal{U}\) under $\mathcal{R}$ is a monoid under the binary
  operation defined by \(\overline{A} \boxplus \overline{B} =
  \overline{A + B}\), where \(\overline{A}\) denotes the equivalence
  class of \(A \in 2^{\mathcal{U}}\).

\end{theorem}

\begin{proof}

  If \(\overline{A_1} = \overline{A_2}\) and \(\overline{B_1} =
  \overline{B_2}\), where \(A_i, B_i \in \mathcal{U}\), then \(A_1 \sim
  A_2\) and \(B_1 \sim B_2\).  By Theorem~\ref{monoidlemma1} part (i),
  \(A_1 + B_1 \sim A_2 + B_2\), so that \(\overline{A_1 + B_1} =
  \overline{A_2 + B_2}\).  Therefore, the binary operation $\boxplus$
  in \(\mathcal{U}/\mathcal{R}\) is well-defined (i.e., it is
  independent of the choice of equivalence-class representatives).  It
  is associative since \(\overline{A + (B + C)} = \overline{A}
  \boxplus (\overline{B} \boxplus \overline{C}) = \overline{A}
  \boxplus (\overline{B + C}) = \overline{(A + B) + C} = (\overline{A
    + B}) \boxplus \overline{C} = (\overline{A} \boxplus \overline{B})
  \boxplus \overline{C}\).

  The identity element is \(\overline{\mathbf{0}}\), the equivalence
  class of all systems that are always up, since \(\overline{A}
  \boxplus \overline{\mathbf{0}} = \overline{A + {\mathbf{0}}} =
  \overline{A}.\) Therefore, \((\mathcal{U}/\mathcal{R}, \boxplus)\)
  is a monoid. \qedhere

\end{proof}

An analogous theorem can also be stated in respect of the $\times$
operator, with an identity element $\overline{\mathbf{1}}$, the
equivalence class of systems that are always up.  The proof is exactly
similar and is thus omitted.

\index{direct product}
\begin{theorem} \label{monoidtheorem2}

  Let \(\mathcal{U}\) be the set of all systems, and the relation
  \(\mathcal{R}(\sim)\) be defined as in Theorem~\ref{monoidlemma1}.
  Then the set \(\mathcal{U}/\mathcal{R}\) of all equivalence classes
  of \(\mathcal{U}\) under $\mathcal{R}$ is a monoid under the binary
  operation defined by \((\overline{A}) \boxtimes (\overline{B}) =
  \overline{A \times B}\), where \(\overline{A} \subseteq
  \mathcal{U}\) denotes the equivalence class of \(A\).

\end{theorem}

Therefore, we have a second monoid, \((\mathcal{U}/\mathcal{R},
\boxtimes)\).

\begin{remark} \label{systemclasses}

  Note that \(\mathcal{U}/\mathcal{R}\) is a set of \emph{classes of
    systems}, rather than of systems themselves.  Therefore,
  \(\mathcal{U}/\mathcal{R} \subseteq 2^{\mathcal{U}}\), and any
  \(\overline{A} \subseteq \mathcal{U}\) (or: \(\overline{A} \in
  2^{\mathcal{U}}\)), where \(A \in \mathcal{U}\), and where
  \(\overline{A}\) denotes the equivalence class of $A$ by worst-case
  fault tolerance.

\end{remark}

It is also possible to re-work Theorem~\ref{monoidtheorem1},
considering the best-case fault tolerance.  To show how, we first
state the analogue of Theorem~\ref{monoidlemma1}.

\begin{theorem} \label{monoidlemma2}

  If we define an equivalence relation \(\mathcal{R} (\sim)\) on
  systems by best-case fault tolerance, such that \(A \sim B
  \longrightarrow \Phi_{best}(A) = \Phi_{best}(B)\), then

  \(A_1 \sim B_1\) and \(A_2 \sim B_2\) imply \(A_1 + A_2 \sim B_1 +
  B_2\).

\end{theorem}

\begin{proof}

  We note from Theorem~\ref{theoremfaulttolerance} $(iii)$ that
  \(\Phi_{best}(A_1 + A_2) = \Phi_{best}(A_1) + \Phi_{best}(A_2)\).
  If \(A_1 \sim B_1\) and \(A_2 \sim B_2\), then \(\Phi_{best}(A_1) +
  \Phi_{best}(A_2) = \Phi_{best}(B_1) + \Phi_{best}(B_2)\).  This
  latter expression is \(\Phi_{best}(B_1 + B_2)\), giving us the
  result. \qedhere

\end{proof}

Note that the analogue of Theorem~\ref{monoidlemma1}(ii) is not
generally true---if \(\mathcal{R}(\sim)\) denotes an equivalence
relation by best-case fault-tolerance, then \(A_1 \sim B_1\) and \(A_2
\sim B_2\) do not imply \(A_1 \times A_2 \sim B_1 \times B_2\).

Therefore, we can re-state Theorem~\ref{monoidtheorem1} (but not
Theorem~\ref{monoidtheorem2}) using the new definition of
$\mathcal{R}$.  The statement and proof run just as previously,
however, so we do not belabor the point.

We may summarize the results of this section, however, to note that
there are three types of equivalence-class monoids so obtained:

\begin{itemize}

\item the monoid of worst-case equivalence classes under direct sums;

\item the monoid of worst-case equivalence classes under direct
  products; and

\item the monoid of best-case equivalence classes under direct sums.

\end{itemize}

\section{The Semiring of Fault-Tolerance Equivalence
  Classes} \label{eqclasssemirings}

It has been shown previously in Theorems~\ref{monoidtheorem1}
and~\ref{monoidtheorem2} that there exist monoids
\((\mathcal{U}/\mathcal{R}, \boxplus)\) and
\((\mathcal{U}/\mathcal{R}, \boxtimes)\) on the equivalence classes of
systems by worst-case fault tolerances, and that these are commutative
monoids.  

It is easily seen that the other two conditions for a
semiring~\cite{Golan1999} (see Section~\ref{monoid}) are also
satisfied:

\begin{itemize}

\item $\boxtimes$ distributes over $\boxplus$---for any
  \(\overline{A}, \overline{B}, \overline{C} \in 2^{\mathcal{U}}\), we
    have:
\begin{eqnarray*}
(\overline{A} \boxtimes \overline{B}) \boxplus (\overline{A} \boxtimes \overline{C}) \ & \ = (\overline{A \times B}) \boxplus (\overline{A \times C}) \\
& \ = \overline{(A \times B) + (A \times C)} \\
& \ = \overline{A \times (B + C)} \\
& \ = \overline{A} \boxtimes (\overline{B} \boxplus \overline{C})
\end{eqnarray*}

\item \(\overline{\mathbf{0}} \boxtimes \overline{A} =
  \overline{\mathbf{0} \times A} = \overline{\mathbf{0}} =
  \overline{A} \boxtimes \overline{\mathbf{0}}\), for all
  \(\overline{A} \subseteq \mathcal{U}\)

\end{itemize}

Therefore, we can consider a different semiring, \((2^{\mathcal{U}},
\boxplus, \boxtimes)\), and it turns out that as with \((\mathcal{U},
+, \times)\), this too is \emph{zerosumfree}, \emph{entire},
\emph{simple}, and \emph{commutative} (compare with the corresponding
Remarks~\ref{remark1},~\ref{remark2},~\ref{remark3},
and~\ref{remark4}).

It is further possible to define a partial ordering relation (denoted
by the symbol $\lesssim$, for example) comparing the fault tolerances
of different classes of systems.  All of Section~\ref{partialordering}
can thus be repeated with $2^{\mathcal{U}}$ in place of $\mathcal{U}$,
$\overline{A}$ and such in place of $A$, and $\lesssim$ in place of
$\preccurlyeq$.

\section{Acknowledgements} \label{ack}

The author would like to thank Ted Herman and Sukumar Ghosh for their
wholesome encouragement of his research in this area, and for
important feedback at the early stages.

\end{document}